\documentclass[letterpaper]{amsart}
\usepackage[english]{babel}
\usepackage{amsmath} 
\usepackage{graphicx}
\usepackage[colorlinks]{hyperref}
\usepackage[colorinlistoftodos]{todonotes}
\usepackage{verbatim}

\newtheorem{theorem}{Theorem}[section]
\newtheorem{lemma}[theorem]{Lemma}

\newtheorem{corollary}[theorem]{Corollary}

\theoremstyle{definition}
\newtheorem{definition}[theorem]{Definition}

\newcommand{\E}{\mathop{\bf E\/}}

\title {An Elementary Proof of the Cayley Formula using Random Maps}

\author[S. Hao]{Steven Hao}
\author[A. He]{Andrew He}
\author[R. Li]{Ray Li}
\author[S. Wu]{Scott Wu}

\begin{document}

\begin{abstract}
Cayley's formula states that the number of labelled trees on $n$ vertices is $n^{n-2}$, and many of the current proofs involve complex structures or rigorous computation\cite{Aigner}. We present a bijective proof of the formula by providing an elementary calculation of the probability that a cycle occurs in a random map from an $n$-element set to an $n+1$-element set.
\end{abstract}

\maketitle

\section {Proof of Cayley's Theorem}

\begin {definition} 
For a set $S \subset \mathbb{Z}^{+},$ define an \emph{$S$-map} to be a map $f:S\to S\cup\{0\}$.
\end{definition}

\begin {lemma}
\label{keylemma}
For any finite, non-empty set $S$ of positive integers, let $f$ be a random $S$-map such that for all $i\in S$, 
\begin{enumerate}
\item the $f(i)$ are chosen independently
\item $P\left[f(i) \neq 0\right] = p$
\item when $f(i) \neq 0$, $f(i)$ is selected uniformly at random from $S$.
\end{enumerate}
Then the probability that $f$ has a cycle is $p$. (Note that as defined, $0$ cannot be in a cycle)
\end {lemma}

\begin {proof}
Let $n=|S|$. We proceed by strong induction on $n$. A base case is not necessary.

For the sake of induction, assume the statement is true for all sets with size less than $n$.

Call an element $i \in S$ \textit{good} if $f(i) \neq 0$. 
Let $G$ be the set of good elements and let $k = \lvert G \rvert$.
We claim that for a fixed set $G$ of good elements, $f$ contains a cycle with probability $\frac{k}{n}$.

If $k = n$, then $f$ clearly contains a cycle, so the probability is $1$.

If $k = 0$, $f$ clearly does not contain a cycle, so the probability is $0$.

In all other cases, $k$ is a positive integer less than $n$.
Now define the $G$-map $f':G\to G\cup \{0\}$ induced by $f$ such that $f'(i) = f(i)$ if $f(i)\in G$ and $f'(i)=0$ otherwise. 
Note that $f$ has a cycle if and only if $f'$ has a cycle.
Note also that the $f'(i)$ are independent, and are chosen uniformly from $G$ when $f'(i) \neq 0$.
Furthermore, for every good element $i$, $f(i) \in G$ with probability $\frac{k}{n}$ so $f'(i) \neq 0$ with probability $\frac{k}{n}$.
Thus, by the inductive hypothesis, $f'$, and therefore $f$, has a cycle with probability $\frac{k}{n}$.

Thus, the probability of a cycle is $\E \left[ \frac{\lvert G \rvert}{n} \right]$.
However, for all $i \in S$, $P\left[i \in G\right] = p$, so $f$ has a cycle with probability $p$.

\end {proof}

\begin{corollary}
\label{keycorollary}
The number of cycle-free $S$-maps is precisely $(n+1)^{n-1}$, where $n = |S|$.
\end{corollary}
\begin{proof}
Let $f$ be a random $S$-map, such that the values $f(i)$ are chosen independently and uniformly at random from $S\cup\{0\}.$ By Lemma~\ref{keylemma}, the probability that $f$ has a cycle is $\frac{n}{n+1},$ and thus the probability it is cycle-free is $\frac{1}{n+1}.$ Furthermore, by construction, each of the $(n+1)^n$ total $S$-maps are equally likely to be chosen as $f.$ Thus, it follows that the number of cycle-free $S$-maps is exactly $\frac{1}{n+1}$ of the total number of $S$-maps, or $(n+1)^{n-1}.$

\end{proof}

\begin {theorem}
(Cayley's Formula) For any positive integer $n$, the number of trees on $n$ labeled vertices is exactly $n^{n-2}$. 
\end {theorem}

\begin{proof} If $n = 1$ the proof is trivial. Assume $n \ge 2$.

Let $[n]$ denote the set $\{1,\dots, n\}.$ 

By Corollary~\ref{keycorollary} there are $n^{n-2}$ cycle-free $[n]$-maps.


Consider the following mapping from $[n-1]$-maps to graphs on $n$ vertices labeled $0,1,2,\dots, n-1$: the image of an $[n-1]$-map $f$ is the graph with an edge between $i$ and $f(i)$ for each $i$ in $[n-1]$ (possibly with double edges or self loops).
We claim this induces a bijection between cycle-free $[n-1]$-maps and trees labeled with $0,1,2,\dots, n-1$: the pre-image of a labeled tree is the map which associates each vertex $i \neq 0$ with the second vertex on the (unique) shortest path from $i$ to $0$. 

Note that the image of any $[n-1]$-map $f$ is cycle-free if and only if $f$ is cycle-free: if $i,f(i),\dots,f^k(i)$ is a cycle in $f,$ then the vertices corresponding to those indices will also form a cycle in the image of $f.$ Similarly, if we have a cycle consisting of vertices $v_1,v_2,\dots, v_n$ in the image of $f,$ then we must either have $f(v_1) = v_2, f(v_2)=v_3 \dots, f(v_n) = v_1,$ or $f(v_2) = v_1, f(v_3) = v_2,\dots, f(v_1) = v_n,$ and in either case, $f$ has a cycle.

It follows that the image of a cycle-free $[n-1]$-map is a tree, as it has $n-1$ edges and is cycle-free. On the other hand, it also follows that the pre-image of a tree is a cycle-free $[n-1]$-map, so this mapping indeed induces a bijection between cycle-free $[n-1]$-maps and trees on $n$ vertices.




It follows that the number of labeled trees on $n$ vertices is equal to the number of cycle-free $[n-1]$-maps, which by Corollary~\ref{keycorollary} is $n^{n-2}.$

\end{proof}

\section {Generalizations}

\begin {lemma}
\label{keyextension}
Consider a finite, non-empty set $S$ of positive integers with size $n$.
Let $\pi$ be a probability distribution over $S$.
Let $f$ be a random $S$-map, such that the values $f(i)$ are independently chosen so that 
with probability $p_i$, $f(i)$ is chosen from $S$ according to $\pi$, and otherwise $f(i) = 0$,
for some $p_i \in [0, 1]$. 

Then, $f$ has a cycle with probability $\sum_{i \in S} p_i\pi(i)$.

In particular, if the $p_i=p$ for all $i$, then the probability of a cycle is $p$.

Note that this probability equals the expected number of fixed points of $f$.

(Note that as defined, $0$ cannot be in a cycle)
\end {lemma}

\begin {proof}
We proceed by strong induction on $n$. A base case is not necessary.

For the sake of induction, assume the statement is true for sets with size less than $n$.

Call an element $i$ of $S$ \textit{good} if $f(i) \neq 0.$ 
Let $G$ be the set of good elements. 
Let $k = \lvert G \rvert$.
Let $q = P\left[f(1) \in G \vert f(1) \neq 0\right] = \sum_{i \in G} \pi(i)$. 
We claim that for a fixed set $G$ of good elements, $f$ contains a cycle with probability $q$.

If $k = n$, the probability is $1$.

If $k = 0$, the probability is $0$.

Now suppose that $0 < k < n$. 
Define the $G$-map $f'$ induced by $f$ such that $f'(i) = f(i)$ if $f(i)\in G$ and $f'(i)=0$ otherwise. 
Note that $f$ has a cycle if and only if $f'$ has a cycle.
Note that the $f'(i)$ are independent and identically distributed for $i \in G$.
Furthermore, for every good element $i,$ $f(i) \in G$ with probability $q$ so $f'(i) \neq 0$ with probability $q$.
Thus, by the inductive hypothesis, $f'$, and therefore $f$, has a cycle with probability $q = \sum_{i \in G} \pi(i)$.

Note that for all $i \in S$, $i$ is good with probability $p_i$. 
Then, by linearity of expectation, the probability of a cycle is 
$\sum_{i \in S} p_i\pi(i)$ and the induction is complete.

\end {proof}
\section{Acknowledgements}

We would like to thank Paul Christiano, Yan Zhang, Po-Shen Loh, and others for their guidance. 

\bibliographystyle{plain}
\bibliography{cayley}
\end{document}